\documentclass[12pt,a4paper,psamsfonts]{amsart}
\usepackage{amssymb,amscd,amsxtra,calc}
\usepackage{cmmib57}

\theoremstyle{plain}
    \newtheorem{thm}{Theorem}[section]

    \newtheorem{theorem}[thm]{Theorem}

\theoremstyle{definition}

\theoremstyle{remark}

\title{\'Etale dynamical systems and topological entropy}

\author{Tuyen Trung Truong}
\address{School of Mathematical Sciences, The University of Adelaide, Adelaide SA 5005, Australia}
\email{tuyen.truong@adelaide.edu.au}
\thanks{The author was supported by Australian Research Council grants DP120104110 and DP150103442.}
\date{\today}
\begin{document}
\begin{abstract}
In this paper, we consider two questions about topological entropy of dynamical systems. We propose to resolve these questions by the same approach of using \'etale analogs of topological and algebraic dynamical systems.

The first question is to define topological entropy for a topological dynamical system $(f,X,\Omega )$. The main idea is to make use - in addition to invariant compact subspaces of $(X,\Omega )$ - of compactifications of \'etale covers $\pi :(f',X',\Omega ')\rightarrow (f,X,\Omega )$, that is $\pi \circ f'=f\circ \pi$ and the fibers of $\pi $ are all finite. We prove some basic results and propose a conjecture, whose validity allows us to prove further results. 

The second question is to define topological entropy for algebraic dynamical systems, with the requirement that it should be as close to the pullback on cohomology groups as possible. To this end, we develop an \'etale analog of algebraic dynamical systems. 

\end{abstract}
\maketitle

\section{The questions and main ideas}
The synthesis of this paper is to propose using ideas from \'etale topology to resolve two seemingly far away questions. Question 1: What is the topological entropy of a topological dynamical system on non-compact topological spaces? Question 2: What is the topological entropy of an algebraic dynamical system, in relation to Gromov - Yomdin's theorem?  

In this paper we study the topological entropy of dynamical systems $(f,X,\Omega )$, that is maps $f:X\rightarrow X$ continuous with respect to a topology $\Omega$ on $X$. In case $(X,\Omega )$ is compact, the topological entropy has been defined in the classical work of Adler, Konheim and McAndrew. When $(X,\Omega )$ is not compact, there have been many proposals, tailored for special cases such as when $(X,\Omega )$ is a homogeneous space, when it is a metric space or when the map $f$ is uniformly continuous. 

Given that topological entropy is abound in dynamical systems and that its idea is applied also to fields outside of mathematics (e.g. to DNA sequences \cite{koslicki}), it is important to make it straight what topological entropy should mean from a purely topological view that it should represent the complexity of general topological dynamical systems, when we do not add any further conditions or restrictions on the dynamical systems (such as smooth and so on). We study this question in the first section to come, by analysing a minimal list of properties that entropy should satisfy. Choosing compact topological spaces as the base case, we propose that entropy should be related to the behaviour of the map on compact invariant subspaces of $(X,\Omega )$ as well as on compactifications of \'etale covers $\pi :(f',X',\Omega ')\rightarrow (f,X,\Omega )$. We say that a dynamical system $(f',X',\Omega ')$ is an \'etale cover of $(f,X,\Omega )$ if there is a surjective continuous map $\pi :X'\rightarrow X$ with {\bf finite fibers} so that $\pi \circ f'=f\circ \pi$. This approach is, to our knowledge, not mentioned in previous works. We prove basic properties for the entropy so defined. 

When $(X,\Omega )$ is moreover a compact Riemannian manifold, there is a conjecture under the name of Entropy Conjecture (first stated by Shub and extended by Katok, see \cite{katok}), relating entropy to the pullback of $f$ on cohomology groups. While this conjecture fails in the general case, Gromov and Yomdin proved it for the case $(X,\Omega )$ is compact K\"ahler and $f:X\rightarrow X$ is holomorphic. Dinh and Sibony generalised the quantity in the Gromov - Yomdin's theorem to the case of meromorphic maps and more generally correspondences of compact K\"ahler manifolds (earlier Russakovskii and Shiffman proved the same result on complex projective spaces). These are now commonly called dynamical degrees.  We generalised the dynamical degrees to correspondences on algebraic varieties over an arbitrary field, by using algebraic cycles modulo numerical equivalence. Adopting the idea of using the graphs of the maps in the proof of Gromov's inequality, Friedland defined topological entropy for meromorphic maps of compact K\"ahler manifolds, and Dinh - Sibony extended this to correspondences of compact K\"ahler manifolds. The Gromov - Yomdin's theorem holds for some general classes of meromorphic maps of compact K\"ahler surfaces, however it fails in many examples. Over non-Archimedean fields, the study of morphisms of $\mathbb{P}^1$ is furnished by lifting it to the Berkovich space $P^1$, and this is a very active field. However, it was shown by Favre and Rivera-Letelier that the Yomdin's inequality fails in this case, that is the topological entropy while still be  always smaller than the Gromov - Yomdin's quantity, may be strictly smaller.   

In the second section to come, we propose a way to resolve this, by developing an algebraic \'etale analog of dynamical systems. This idea is inspired by several different facts. First, while over fields of positive characteristic  resolution of singularities by modifications are not available, by results of de Jong  resolutions of singularities by alterations always exist. Alterations are generically finite regular morphisms, hence are more general than \'etale morphism. Second, the dynamical degrees have been shown by Dinh and Nguyen to be an invariant for semi-conjugacies of rational maps of compact K\"ahler manifolds by a generically finite rational map, and we generalised this fact to correspondences over an arbitrary field. A generically finite rational map is more general than an alteration, since it may not be a regular morphism. Last, a result by Esnault and Srinivas showed that for an automorphism of a projective surface over a field of arbitrary characteristic, the dynamical degrees can be computed via the pullback of the map on $l$-adic cohomology groups. Here, again we see the appearance of \'etale maps. 

An \'etale view of dynamical systems requires that if we start with a dynamical system $f:X\rightarrow X$, we also need to consider the pullback dynamical systems $\pi ^*f:Z\rightarrow Z$, where $\pi :Z\rightarrow X$ is a generically finite meromorphic map. Here, even if $f$ is a rational map, the pullback $\pi ^*f$  in most cases is not a rational map but only a correspondence. Our previous work on relative dynamical  degrees of correspondences over an arbitrary field is applicable for this purpose. In case $X=\mathbb{P}^k$ is a projective space, since any projective variety $Z$ has such a map $\pi :Z\rightarrow X$, we are forced to consider correspondences over all algebraic varieties even if we are only concerned with dynamics over $\mathbb{P}^k$. Note also that since here we have the specific purpose of matching up the topological entropy and the dynamical degrees, we have to consider very special topologies when defining entropy. Every field $K$ becomes non-Archimedean under the trivial metric $d(x,y)=1$ if $x\not= y$ and $=0$ if $x=y$, and the topology is then the discrete one. From the general philosophy that the dynamics with respect to the discrete topology should have the most complexity, we propose that the corresponding Berkovich spaces may be useful in defining the topological entropy. 

{\bf Remark.} The approaches proposed for the topological entropy of topological and algebraic dynamical systems, as proposed above, have the common feature of using appropriate dynamical systems dominating the given dynamical system. These are very resemblant of the idea used by Grothendieck, that of replacing open subsets of $X$ by \'etale covers lying over $X$, to define \'etale topology. Hence, we may call these the \'etale dynamical systems. 

{\bf Acknowledgments.} We would like to thank Jan Boronski and Nessim Sibony for pointing out a mistake in a previous version of the paper. We thank Charles Favre for useful discussions. The valuable comments by Tien-Cuong Dinh and Nessim Sibony are also very much appreciated.    

\section{Topological \'etale dynamical systems and entropy}

Let $(X,\Omega )$ be a topological space, need not be compact. Let $f:(X,\Omega )\rightarrow (X,\Omega )$ be a continuous map. One important characterisation for the complexity of the dynamical behaviour of $f$ is its topological entropy. The entropy often dictates other important features of the dynamical system, such as invariant measures and the  growth of periodic points (see e.g. \cite{walters}). 

When $(X,\Omega )$ is compact, the topological entropy was defined by Adler, Konheim and McAndrew \cite{adler-konheim-mcandrew}, which relates to the measure theoretic entropy of Kolmogorov - Sinai \cite{kolmogorov, sinai} via the variational principle.

When $(X,\Omega )$ is Hausdorff but not compact, Bowen \cite{bowen} gave a definition of topological entropy by using the idea of Hausdorff dimension. He presented, however, a simple example due to L. Goodwyn of a map whose entropy, according to his definition, is $\infty$. There have been other definitions (see \cite{hasselbaltt-nitecki-propp} for relations between some common approaches), which use for example the Stone - Cech compactification in Hofer \cite{hofer} or the compact subsets of $X$ invariant by $f$ in Canovas - Rodriguez \cite{canovas-rodriguez}. Canovas and Rodriguez also mentioned in passing the generalisation to the case $X$ is not Hausdorff. Liu-Wang-Wei \cite{liu-wang-wei} developed in more details the definition in \cite{canovas-rodriguez} as follows. Let $f:(X,\Omega )\rightarrow (X,\Omega )$ be a continuous map, where $(X,\Omega )$ does not assume to have any special property. We define 
\begin{equation}
h_{CR}(f,\Omega ):=\sup \{h_{top}(f|_K)|~K~\mbox{is compact in }(X,\Omega ),~f(K)\subset K\}.
\label{Equation0}\end{equation}
 In the above, if there is no compact subset $K$ such that $f(K)\subset K$, we define $h_{CR}(f,\Omega ):=0$. While seeming reasonable and having various good properties, this definition has the drawback that the compact subsets of $(X,\Omega )$ may not be abundant enough to represent reasonably the complexity of a continuous map. Our new definition resolves this by adding into account also the compactifications (if exist) appropriate dynamical systems $(f,'X',\Omega ')$ having $(f,X,\Omega )$ as a topological quotient. 
 
 In this paper we propose, for a continuous map $f:(X,\Omega )\rightarrow (X,\Omega )$, a new notion of topological entropy $h_{\Omega}(f)$, which works for an arbitrary topology. We recall that given two topologies $\Omega ,\Omega '$ on $X$, then $\Omega '$  is finer than $\Omega $, denoted as $\Omega '>\Omega $, if every open set in $(X,\Omega )$ is also an open set in $(X,\Omega ')$. This is the same as that the identity map $\iota _X:(X,\Omega ')\rightarrow (X,\Omega )$, $\iota _X(x)=x$ for all $x\in X$, is continuous. We say that a dynamical system $(f',X',\Omega ')$ is a cover of $(f,X,\Omega )$ if there is a surjective continuous map $\pi :X'\rightarrow X$ so that $\pi \circ f'=f\circ \pi$. In other words, $(f',X',\Omega ')$ is a cover of $(f,X,\Omega )$ if and only if $(f,X,\Omega )$ is a topological quotient of $(f',X',\Omega ')$. We say that a cover $(f',X',\Omega ')$ of $(f,X,\Omega )$ is \'etale if moreover all the fibers of $\pi$ are finite. 

To motivate the definition, let us first describe four reasonable properties the new definition should satisfy, so that it represents reasonably the complexity of a continuous map. 

i) If $(X,\Omega )$ is compact, then $h_{\Omega }(f)=h_{top}(f)$, where the right hand side is the classical one. 

ii) Let $(X,\Omega )$ be embedded as a subspace in another topological space $(Z,\Omega _Z)$, where the map $f$ is the restriction of a continuous map $f_Z:(Z,\Omega _Z)\rightarrow (Z,\Omega _Z)$. Then, the complexity of the dynamical system $(f,X,\Omega )$, being a portion of the dynamical system $(f_Z,Z,\Omega _Z)$, is not more than that of $(f_Z,Z,\Omega _Z)$. Hence we should have $h_{\Omega }(f)\leq h_{\Omega _Z}(f_Z)$. 

iii) Let $(f',X',\Omega ')$ be an \'etale cover of $(f,X,\Omega )$. Then it is justified to view that the complexity of $(f',X',\Omega ')$ is not less than that of $(f,X,\Omega )$. Therefore, we should have $h_{\Omega '}(f')\geq h_{\Omega }(f)$. In particular, if $(f_Z,Z,\Omega _Z)$ is a compactification of $(f',X',\Omega ')$, that is $(Z,\Omega _Z)$ is a compact topological space containing $(X',\Omega ')$ as as dense subspace, then 
\begin{equation}
h_{\Omega }(f)\leq h_{top}(f_Z).
\label{Equation}\end{equation}

iv) The dynamical systems $(f_Z,Z,\Omega _Z)$ in iii) determines the complexity of $(f,X,\Omega )$ uniquely. Therefore, $h_{\Omega}(f)$ should be the largest number satisfying conditions i), ii) and iii). 

{\bf Remark.} In the situation of Requirement iii), if both $(X',\Omega ')$ and $(X,\Omega )$ are compact metric spaces, then by classical results we have that $h_{top}(f')=h_{top}(f)$. This is an additional justification for that the complexity of $(f,X,\Omega )$ should be determined by the complexity of its \'etale covers. 

A special case of \'etale covers is that when $f'=f$, $X'=X$ and $\Omega '$ is finer than $\Omega$. In this case, the justification for Requirement iii) is clearer: Since $f'=f$ on the set level, the only factor that determines the complexity is the topology.  

We now proceed to defining a topological entropy satisfying Requirements i) and ii), while satisfying Requirement iii) in some special cases. If Conjecture 1, to be stated below, holds, then Requirement iii) is also valid in general. For a dynamical system $(f,X,\Omega )$, we define by $\mathcal{C}(\Omega ,f)$ the set of all compactifications $(f_Z,Z,\Omega _Z)$. By definition, the latter are dynamical systems $(f_Z,Z,\Omega _Z)$ where $(Z,\Omega _Z)$ is a compact topological space and contains $(X,\Omega )$ as a dense subspace, and $f=f_Z|_X$. Define 
\begin{eqnarray*}
\mathcal{A}(\Omega ,f)&:=&\{(f_Z,Z,\Omega _Z) |~(f_Z,Z,\Omega _Z)\in  \mathcal{C}(\Omega ',f'),\\
&&~(f',X',\Omega ')~ \mbox{ is an \'etale cover of } ~(f,X,\Omega ) \}.
\end{eqnarray*}
  We note that $\mathcal{A}(\Omega ,f)$ is always non-empty because the discrete topology $\Omega _0$ on $X$, being normal Hausdorff, can be embedded as a dense subspace in its Stone - Cech compactifiaction where the continuous map $(f,X,\Omega _0)$ extends uniquely. On the other hand, the dynamical system $(f,X,\Omega )$  may not have any compactifications. Because of the similarity to the ideas used in \'etale topology, we call  $\mathcal{A}(\Omega ,f)$ the (topological) \'etale dynamical system associated with $(f,X,\Omega )$. An algebraic analog will be used in the next section.   

Our topological entropy is
\begin{equation}
h_{\Omega }(f):=\max\{h_{CR}(f,\Omega ),\inf _{(f_Z,Z,\Omega _Z)\in \mathcal{A}(\Omega ,f)}h_{top}(f_Z)\}. 
\label{Equation1}\end{equation}
The first term in the right hand side of Equation (\ref{Equation1}) may be viewed as the internal entropy of $f$ (by Requirement ii) it is bounded from above by $h_{\Omega }(f)$), while the second term may be viewed as the external entropy of $f$. If Requirement iii) is to be satisfied then the second term in the right hand side of Equation (\ref{Equation1}) should bound from above $h_{\Omega}$, which would be the case if it is not smaller than the first term. We suspect that this is the case, and propose Conjecture 1 below, whose validity allows us to prove that $h_{\Omega}(f)$ in fact satisfies Requirement iii) as desired. Without Conjecture 1, we will prove below that Requirement iii) is satisfied in some special cases. 

The topological entropy, defined by Equation (\ref{Equation1}), has the following good properties. 
\begin{theorem} Let $f:(X,\Omega )\rightarrow (X,\Omega )$ be a continuous map. 

1) (Compact space.) If $(X,\Omega )$ is compact, then $h_{\Omega}(f)=h_{top}(f)$.

 2) (Invariant subspace.) Let $(Y,\Omega |_Y)$ be a subspace of $(X,\Omega )$. Assume that $f(Y)\subset Y$. Then $h_{\Omega |_Y}(f|_Y)\leq h_{\Omega }(f)$. 

3) (Special \'etale cover.) Assume that $(f',X',\Omega ')$ is an \'etale cover of $(f,X,\Omega )$ via a surjective continuous map $\pi :(X',\Omega ')\rightarrow (X,\Omega )$. Assume moreover that $\pi$ is proper, that is the preimage of any compact set is a compact set. Then $h_{\Omega }(f)\leq h_{\Omega '}(f')$. 

4) (Iterates.) For all $n\in \mathbb{N}$, we have $h_{\Omega}(f^n)\leq nh_{\Omega}(f)$. 

5) (Product.) We have $$\max \{h_{\Omega _1}(f_1),h_{\Omega _2}(f_2)\}\leq h_{\Omega _1\times \Omega _2}(f_1\times f_2)\leq h_{\Omega _1}(f_1)+h_{\Omega _2}(f_2).$$
 
\label{Theorem1}\end{theorem}
\begin{proof}
1) If $(X,\Omega )$ is compact, then $(f,X,\Omega )$ itself belongs to $\mathcal{A}_{\Omega ,f}$. Hence, 
\begin{eqnarray*}
h_{top}(f,\Omega )\geq \inf _{(f_Z,Z,\Omega _Z)\in \mathcal{A}(\Omega ,f)}h_{top}(f_Z).
\end{eqnarray*}
On the other hand, in this case $h_{CR}(f,\Omega )=h_{top}(f,\Omega )$. Hence we have 1). 

2) If $K$ is a compact subset of $(Y,\Omega |_Y)$, then $K$ is also a compact subset of $(X,\Omega )$. Hence, by definition, we have $h_{CR}(f|_Y,\Omega |_Y)\leq h_{CR}(f,\Omega )$. 

If $(f,X,\Omega )$ has a compactification $(f_Z,Z,\Omega _Z)$, then $(f|_Y,Y,\Omega |_Y)$ has a compactification $(f_Z|_{Y'},Y',\Omega _Z|_{Y'})$, where $Y'$ is the closure of $Y$ in $Z$. By properties of topological entropy for compact  spaces, we have $h_{top}(f_Z)\geq h_{top}(f_Z|_{Y'})$. From this, we have 2). 

3) Because $(f',X',\Omega ')$ is an \'etale cover of $(f,X,\Omega )$, whenever $(f_Z,Z,\Omega _Z)$ is in $\mathcal{A}(\Omega ',f')$ it also belongs to $\mathcal{A}(\Omega ,f)$. Therefore, 
\begin{eqnarray*}
\inf _{(f_Z,Z,\Omega _Z)\in \mathcal{A}(\Omega ,f)}h_{top}(f_Z)\leq \inf _{(f_{Z'},Z',\Omega _{Z'})\in \mathcal{A}(\Omega ',f')}h_{top}(f_{Z'}).
\end{eqnarray*}
The assumption that $\pi$ is proper implies that $h_{CR}(f,\Omega )\leq h_{CR}(f',\Omega ')$. Hence, we obtain 3). 

The proofs of parts 4) and 5) are similar. 
 \end{proof}

Now we discuss a conjecture in connection to Requirement iii).

{\bf Conjecture 1.} Let $(f',X',\Omega ')$ be an \'etale cover of $(f,X,\Omega )$. Assume that $(X,\Omega )$ is compact, and $(f_Z,Z,\Omega _Z)$ is a compactification of $(f',X',\Omega ')$. Then we have $h_{top}(f)\leq h_{top}(f_Z)$. 

{\bf Remark.} While this conjecture seems reasonable, at the moment we have no idea about its validity.  Note that if in the above $(f',X',\Omega )$ is merely a cover of $(f,X,\Omega )$ then Conjecture 1 fails in general. A simple counterexample, mentioned in \cite{canovas-rodriguez}, is the following: $X'=\mathbb{R}$ and $X=S^1$ with the usual topologies, $f'(x)=2x$, $f(x)=2x$, and $\pi :X'\rightarrow X$ is the universal covering map. Then $h_{top}(f)=\ln (2)>0$, while $(f',X',\Omega ')$ has a compactification whose topological entropy is $0$. Here, all fibers of $\pi$ are infinite.

The following result is proven under the assumption that Conjecture 1 holds. 

\begin{theorem}
Assume that Conjecture 1 holds. Then, we have 
$$h_{CR}(f,\Omega )\leq \inf _{(f_Z,Z,\Omega _Z)\in \mathcal{A}(\Omega ,f)}h_{top}(f_Z).$$
As a consequence, we obtain the following property of $h_{\Omega}(f)$:

(Topological quotient.) If $(f',X',\Omega ')$ is an \'etale cover of $(f,X,\Omega )$, then $h_{\Omega }(f)\leq h_{\Omega '}(f')$.
\label{Theorem2}\end{theorem}
\begin{proof}
Let $(f_Z,Z,\Omega _Z)$ be an element of $\mathcal{A}(\Omega ,f)$. Then $(f,X,\Omega )$ is a topological quotient of some $(f',X',\Omega ')$ via a surjective continuous map $\pi :(X',\Omega ')\rightarrow (X,\Omega )$, whose fibers are all finite. Let $K$ be a compact subset in $(X,\Omega )$ so that $f(K)\subset K$. Let $K'$ be the closure in $(Z,\Omega _Z)$ of $\pi ^{-1}(K)$. Then $(f_Z|_{K'},K',\Omega _Z|_{K'})$ is a compactification of $(f'|_{\pi ^{-1}(K)},\pi ^{-1}(K),\Omega '|_{\pi ^{-1}(K)})$. Since $\pi$ is surjective and has finite fibers, it follows that $(f_Z|_{\pi ^{-1}(K)},\pi ^{-1}(K),\Omega _Z|_{\pi ^{-1}(K)})$ is an \'etale cover of $(f|_K,K,\Omega |_K)$.  If Conjecture 1 holds, we have $h_{top}(f|_K)\leq h_{top}(f_Z|_{K'})$, and the latter is not greater than $h_{top}(f_Z)$. Taking supremum over all such $K$'s, we obtain $h_{CR}(f)\leq h_{top}(f_Z)$. Taking infimum over all such $(f_Z,Z,\Omega _Z)$, we obtain the desired result.

\end{proof}

{\bf Comparisons with several previous works.}

Assume that $(X,\Omega )$ is Hausdorff. The definition by Hofer \cite{hofer} is $h_{Hof}(f):=h_{top}(f_{SC(X,\Omega )})$ where $SC(X,\Omega )$ is the Stone - Cech compactification of $(X,\Omega )$ and $f_{SC(X,\Omega )}:SC(X,\Omega )\rightarrow SC(X,\Omega )$ is the extension of $f:(X,\Omega )\rightarrow (X,\Omega )$. It can be checked easily that $h_{Hof}(f)\geq h_{\Omega }(f)$. However, $h_{Hof}$ is usually very big to be useful, since it is bigger than the topological entropy of any compactification of $(f,X,\Omega )$, even in simple situations. For example, let $X=\mathbb{N}$ and $f(x)=x+1$. Then it is shown in \cite{hofer} that $h_{Hof}(f)=\infty$. On the other hand, $\mathbb{N}$ is a subset of $\mathbb{C}$, and the map $f$ is the restriction of the map $f':\mathbb{C}\rightarrow \mathbb{C}$, where $f'(x)=x+1$. The topology $\Omega$ is the usual topology on $\mathbb{C}$. The map $f'$ has a compactification in $Z=\mathbb{P}^1$, and the topological entropy of the extension $f_Z$ is $0$.  It is also shown in \cite{canovas-rodriguez} that $h_{CR}(f,\Omega )=0$. Therefore, we have $h_{\Omega }(f)=0$. Since the map $f$ is really simple in this case, the value $0$ is more reasonable to be assigned as the topological entropy of $f$. By the same arguments, we can see that the entropy for the example of Goodwyn, as mentioned in \cite{bowen}, is also $0$. The same applies also for another example, $f(x)=2x$ from $\mathbb{R}\rightarrow \mathbb{R}$, whose entropy according to Hofer is again $\infty$. According to another definition by Bowen \cite{bowen1}, the entropy for this same map should be $\geq \ln (2)$, since the map $f$ is homogeneous. However, since there is no priority for which compactification of $\mathbb{R}$ we should choose, in particular we can choose $\mathbb{P}^1$, it seems more reasonable for the entropy of this map to be $0$. 

The definition by Canovas - Rodriguez \cite{canovas-rodriguez} and Liu-Wang-Wei \cite{liu-wang-wei} on the other hand may be too small sometimes to be useful. For example, let $f:X\rightarrow X$ be any map. Let $\Omega _0$ be the discrete topology on $X$. Since any compact subset $K$ of $(X,\Omega _0)$ has only a finite number of elements, it follows that $h_{CR}(f,\Omega _0)=0$.  If Conjecture 1 holds and there is a topology $\Omega $ on $X$ such that $(X,\Omega )$ is compact Hausdorff and $f$ is continuous with respect to $\Omega$, then by Theorem \ref{Theorem2} we obtain $h_{\Omega _0}(f)\geq h_{top}(f,\Omega )$, and hence is positive provided that $h_{top}(f,\Omega )>0$.  For another example, consider any map $f:\mathbb{R}\rightarrow \mathbb{R}$ satisfying $f(x)>x$, and $\Omega $ is the usual topology on $\mathbb{R}$. Then there is no compact set $K\subset \mathbb{R}$ which is invariant by $f$, and hence $h_{CR}(f,\Omega )=0$.  Note that by definition we always have $h_{\Omega }(f)\geq h_{CR}(f,\Omega )$.   

\section{Algebraic \'etale dynamical systems and entropy} In this subsection, we develop an algebraic analog of the topological \'etale dynamical systems used in the previous section to study topological entropy of algebraic dynamical systems. Since our purpose here is to seek for a topological entropy for which the Gromov - Yomdin 's theorem holds, we need to choose special topologies. Thus, the definition in this case is apparently different from that in the topological case. It is, however, still interesting to compare the two definitions for a regular morphism $f:X\rightarrow X$ of smooth algebraic varieties.

Let $X$ be a connected compact K\"ahler manifold and $f:X\rightarrow X$ a surjective endomorphism. For $p=0,\ldots ,\dim (X)$, let $\lambda _p(f)$ be the $p$-the dynamical degree, defined as the spectral radius of the pullback on cohomology $f^*:H^{p,p}(X)\rightarrow H^{p,p}(X)$. Gromov \cite{gromov} and Yomdin \cite{yomdin} proved that the topological entropy, $h_{top}(f)$, of $f$ can be computed via the formula
\begin{eqnarray*}
h_{top}(f)=\max _{p=0,\ldots ,\dim (X)}\log \lambda _p(f).
\end{eqnarray*}  
Gromov proved the inequality $h_{top}(f)\leq \max _{p=0,\ldots ,\dim (X)}\log \lambda _p(f)$ by relating topological entropy to the graphs of the iterates of $f$. Friedland \cite{friedland} adopted this to define topological entropy for meromorphic maps, and another definition following Bowen was used by Dinh and Sibony \cite{dinh-sibony2} to study quasi-polynomial maps. These definitions have been developed in more general settings, such as foliations (see \cite{dinh-nguyen-sibony}) and correspsondences (see \cite{dinh-sibony3}).  Guedj \cite{guedj} showed that these two definitions are the same (note that the definitions in \cite{guedj} are a little different from those used by Friedland and Dinh - Sibony). Dinh and Sibony \cite{dinh-sibony10, dinh-sibony1} (and earlier, for the case of complex projective spaces, in \cite{russakovskii-shiffman}) defined dynamical degrees for meromorphic maps of compact K\"ahler manifolds, and proved Gromov's inequality in this general context. Hence, dynamical degrees can be viewed as a generalisation of the Gromov - Yomdin's quantity, which we may call the cohomological entropy of a meromorphic map. However, the Yomdin's inequality fails for meromorphic maps, see \cite{guedj}. 

Next, we consider a dominant polynomial map $f:\mathbb{C}^k\rightarrow \mathbb{C}^k$. When $k=1$, it is a classical result of Brolin, Freie-Lopes-Mane, Lyubich and Tortrat, that the topological entropy of the lift of $f$ to $\mathbb{P}^1$ satisfies the Gromov - Yomdin's theorem. When $k=2$,  Diller, Dujardin and Guedj \cite{diller-dujardin-guedj3}, extending earlier pioneering work by Bedford - Lyubich - Smillie, Fornaess - Sibony, Friedland - Milnor and Douady - Hubbard, showed that there is a smooth compactification $X$ (constructed earlier by Favre and Jonsson \cite{favre-jonsson}) of $\mathbb{C}^2$ for which the topological entropy of the lifting of (some iterates of) $f$ satisfies Gromov - Yomdin's theorem. For higher $k$, the Gromov - Yomdin's theorem is proven for only special classes of polynomial maps. While it is folklore conjectured that a rational map $f:X\rightarrow X$ having some reasonable conditions (in particular, it must be algebraic stable in the sense that the  $(f^n)^*=(f^*)^n$ for all $n\in \mathbb{N}$ as linear maps on the cohomology group $H^*(X)$) will have good dynamical properties, in particular will satisfy the Gromov - Yomdin's theorem, the conjecture is currently known mostly only in dimension $2$ for some general classes of maps and for some special classes of maps in higher dimensions. A common strategy is to find, for a given rational map $f:X\rightarrow X$, a birational model for which the dynamics behave nicely. However, this seems hopeless in dimension at least $3$, and is more so over fields of positive characteristics where resolution of singularities is not available.   

Let $K$ be a field which is  non-Archimedean and algebraically closed. Any rational map $f:\mathbb{P}^1_K\rightarrow \mathbb{P}^1_K$ lifts to a continuous map $f:P^1\rightarrow P^1$, where $P^1$ is the Berkovich space of $\mathbb{P}^1_K$. Studying the dynamics of such maps is a very active field. In this case, however, the Yomdin's inequality still does not hold in general, see \cite{favre-letelier}.  

In previous work \cite{truong}, we have defined (relative) dynamical degrees, for rational maps and more general correspondences, of algebraic varieties (not necessarily smooth or compact) over an arbitrary field, in terms of algebraic cycles modulo numerical equivalences. These can be viewed as the geometric entropy of a rational map. In the algebraic setting, over fields $K$ different from $\mathbb{C}$, there is still no definition of cohomological entropy. There are indications \cite{dinh-nguyen, esnault-srinivas, truong} that one should be willing to work with objects such as \'etale topology  - seemingly further away from dynamical systems - in order to come closer to satisfying Gromov - Yomdin 's theorem. In fact, for automorphisms of compact surfaces, the work in \cite{esnault-srinivas} relates the quantity in Gromov - Yomdin's theorem to $l$-adic cohomology, which suggests that we may even need to deal with stranger objects than \'etale topology. Even when $K=\mathbb{C}$, the fact that the dynamical degrees are invariant under generically finite semi-conjugacies (\cite{dinh-nguyen}, see \cite{truong} for a generalisation to correspondences) is an evidence in support of the philosophy that to understand the dynamics of a rational map $f:X\rightarrow X$, it is useful and necessary to study all rational maps $f':X'\rightarrow X'$ (or even correspondences) which are semi-conjugate to $f$ via a generically finite rational map $\pi :X'\rightarrow X$. (Note that the latter maps are more general than isogenies.) Over fields of positive characteristic, the fact that singularities can be resolved by alterations is an additional support. In other words we should work with an \'etale analog of dynamical systems.  

Following the idea of \'etale topology, we proceed as follows. First, we consider the case where $K=\mathbb{C}$. Let $f:X\rightarrow X$ be a dominant rational map or more generally dominant correspondence, where $X$ is a complex projective. Let $\mathcal{B}(f)$ be the set of all dominant correspondences $f_Z:Z\rightarrow Z$, where $Z$ is a complex projective variety so that $\pi _Z\circ f_Z=d_Z f\circ \pi _Z$, where $\pi _Z:Z\rightarrow X$ is a generically finite rational map and $d_Z$ is a positive integer. (Even if we start with rational maps $f$, we still need to consider correspondences in general.) The \'etale analog of dynamical systems is the dynamics of each element in this family of correspondences $f_Z$. The \'etale analog of the dynamical degrees is the supremum of the (weighted) dynamical degrees of this family $f_Z$, which is in fact the same as that for each individual $f_Z$, see \cite{dinh-nguyen, truong}.  We then define the \'etale analogue of entropy
\begin{eqnarray*}
h_{et}(f):=\sup _{Z\in\mathcal{B}(f)}[ h_{top}(f_Z)-\ln (d_Z)],
\end{eqnarray*}
 where the $h_{top}$ in the right hand side is the topological entropy defined in \cite{dinh-sibony3}. The expression in the bracket in the above equation is the weighted topological entropy of $f_Z:Z\rightarrow Z$, defined in analog to the weighted dynamical degrees of $f_Z$. Note that we can choose $Z=X$, and hence $h_{eta}(f)\geq 0$. Moreover, whenever $\pi _Z:Z\rightarrow X$ is a generically finite rational map, we can always find such a correspondence $f_Z$ (the pullback of $f$ by $\pi _Z$, see \cite{truong} for more details), where $d_Z=\deg (\pi _Z)$. By the results in \cite{dinh-sibony3}, we have that $h_{eta}(f)$ is bounded from above in terms of the dynamical degrees of $f$. More precisely, 
\begin{eqnarray*}
h_{et}(f)\leq \max _{0\leq j\leq \dim (X)}\log \lambda _j(f),
\end{eqnarray*}
where $\lambda _j(f)$'s are the dynamical degrees of $f$. We conjecture that the equality happens. 

Although many things are known for the dynamics of algebraic stable maps on complex projective surfaces, the dynamics of non-algebraic stable  maps are mostly unexplored. The above idea may be used as follows. Let us start with a dynamical system $(f,X)$ on a complex projective surface. If $h_{top}(f)$ is not equal to the Gromov - Yomdin's quantity, then we can try to find a generically finite rational map $\varphi :Z\rightarrow X$ so that the entropy of the correspondence $f_Z=\varphi ^*(f)$ is strictly bigger than the sum of the entropy of $f$ and $\ln (\deg \varphi )$. We may in particular choose $Z$ to be singular, and in this case $\varphi$ may be tried to choose to be birational. 

For an algebraically closed field $K$ of arbitrary characteristic, we can proceed as above. More precisely, for a dominant correspondence $f:X\rightarrow X$ where $X$ is an irreducible algebraic variety, we let $\Gamma _f=\sum _{j}\Gamma _j$ be the irreducible decomposition of the graph of $f$, and $\pi _1,\pi _2:X\times X\rightarrow X$ the two projections. We next define $I_1(f)=\{x\in X:~\dim \pi _1^{-1}(x)\cap \Gamma _f\}\geq 1$.  Let $\Gamma _{\infty}=\{(x_m,i_m)_{m=1}^{\infty}\in X^{\mathbb{N}}:~x_m\notin I_1(f)~\forall ~m, ~(x_m,x_{m+1})\in \Gamma _m \}$. We have a shifting map $\sigma $ on $\Gamma _{\infty}$. Let $\Omega $ be a Hausdorff topology on $X$, finer than the Zariski topology, and assume that $(X,\Omega )$ has a favourite compactification $(\widetilde{X},\widetilde{\Omega})$.  Let $X_{\infty}$ be the closure of $\Gamma _{\infty}$ in $(\widetilde{X},\widetilde{\Omega})^{\mathbb{N}}$. We have a shifting map $\sigma _{\infty}:~X_{\infty}\rightarrow X_{\infty}$. Then, we define
\begin{eqnarray*}
h_{top}(f,X):=h_{top}(\sigma _{\infty}). 
\end{eqnarray*}
Now, we define $\mathcal{B}'(f)$ to be the set of dominant correspsondences $f_Z:Z\rightarrow Z$, where $f_Z:Z\rightarrow Z$ is a correspondence and $Z$ is an algebraic variety so that $\pi _Z\circ f_Z=d_Z f\circ \pi _Z$, where $\pi _Z:Z\rightarrow X$ is a generically finite rational map and $d_Z$ is a positive integer. (Note that in the case of positive characteristic, we probably need to restrict to only the case where $f_Z=\pi _Z^*(f)$ is the pullback of $f$. This is because, in the lack of resolution of singularities, in \cite{truong} we are able to prove that dynamical degrees of $f_Z$ are the multiplication of that of $f_Z$ and the topological degree $d_Z$ only in the special case where $f_Z=\pi _Z^*(f)$.) Finally, we define
\begin{eqnarray*}
h_{et}(f):=\sup _{f_Z\in \mathcal{B}'(f)}[h_{top}(f_Z,Z)-\ln (d_Z)]. 
\end{eqnarray*}
 As before we have that $h_{et}(f,\Omega )\geq 0$. However, in the general case, as mentioned before, there is only the geometric version of dynamical degrees \cite{truong}. 
  
 Assume that $K$ is non-Archimedean. Then a favourite compactification is to use the Berkovich space. In the special case $X=\mathbb{P}^1$, \cite{favre-letelier} showed that the topological entropy of (the lifting) of a rational map $P/Q$ on the Berkovich space $P^1$ is zero if and only if there is a change of coordinates such that $P/Q$ has good reduction. One may start with one such example, and then try to show that there is a generically finite rational map $\varphi $ such that the entropy of the correspondence $f_Z=\varphi ^*(f)$ is strictly bigger than the sum of the entropy of $f$ (which in this case is $0$) and $\ln (\deg \varphi )$. Again, we may try first the case where $Z$ is a singular rational curve, and may use the trivial metric on $K$. If we want to keep $Z$ to be a smooth projective curve, then the use of a rational map of degree $>2$ ,and hence correspondences, is unavoidable in general.  
 
 As suggested by the work of Esnault - Srinivas \cite{esnault-srinivas}, we may need to consider topological entropy of more general objects than \'etale dynamical systems.

\end{document}